\documentclass[11pt,reqno]{amsart}

\usepackage{amsthm,amsmath,amssymb,amsfonts,amscd,mathtools}
\RequirePackage{array}
\RequirePackage{cancel}
\RequirePackage{enumerate}
\RequirePackage{mathtools}
\RequirePackage{tikz-cd}
\usetikzlibrary{positioning}
\usepackage{dynkin-diagrams} 
\RequirePackage{geometry}
\RequirePackage{comment}
\usepackage[pagebackref = true,
            colorlinks = true,
            linkcolor = blue,
            urlcolor  = blue,
            citecolor = blue]{hyperref}
\usepackage{cleveref}
\renewcommand*{\backref}[1]{}
\renewcommand*{\backrefalt}[4]{%
  \ifcase #1
  \or Cited on page #2.
  \else Cited on pages #2.
  \fi
}

\allowdisplaybreaks

\newcommand{\R}{\mathbb R}

\renewcommand{\P}{\mathcal {P}}

\newcommand{\G}{\mathsf {G}}
\newcommand{\K}{\mathsf {K}}

\newcommand{\V}{\mathcal {V}}
\newcommand{\A}{\mathcal A}
\renewcommand{\H}{\mathcal H}
\renewcommand{\S}{\mathcal S}
\newcommand{\X}{\mathcal X}
\renewcommand{\L}{\mathcal L}

\renewcommand{\i}{\operatorname{i}}
\newcommand{\Id}{\operatorname{Id}}
\newcommand{\tr}{\operatorname{tr}}

\newcommand{\GL}{\operatorname{GL}}

\newcommand{\Herm}{{\operatorname{Herm}}}

\theoremstyle{plain}
\newtheorem{theorem}{Theorem}[section]
\newtheorem{proposition}[theorem]{Proposition}
\newtheorem{lemma}[theorem]{Lemma}
\newtheorem{corollary}[theorem]{Corollary}

\newtheorem{introtheorem}{Theorem}

\newtheorem{introconjecture}[introtheorem]{Conjecture}
\newtheorem{introcorollary}[introtheorem]{Corollary}

\theoremstyle{definition}

\theoremstyle{remark}
\newtheorem{remark}[theorem]{Remark}

\newtheorem{introassumptions}[introtheorem]{Assumptions}

\title{The Fino--Vezzoni conjecture on homogeneous spaces}

\author{Joseph Kwong}
\address[J. Kwong]{The University of Queensland, Brisbane, QLD 4072, Australia}
\email{j.kwong@student.uq.edu.au}

\begin{document}

\begin{abstract}
   We prove that a compact discrete quotient of a complex homogeneous space with compact isotropy is K\"ahler whenever it admits both a balanced metric and a pluriclosed metric. Moreover, if its real first Chern class vanishes, then the pluriclosed flow starting from any invariant pluriclosed metric exists for all time and converges smoothly to a flat K\"ahler metric.
\end{abstract}
\maketitle

\section{Introduction}
Let $Y$ be a complex manifold of complex dimension $n \geq 2$. A Hermitian metric $\omega$ on $Y$ is called \textit{balanced} if $d \omega^{n-1} = 0$, and \textit{pluriclosed} if $\partial \overline \partial \omega = 0$. Every
K\"ahler metric is automatically both balanced and pluriclosed. In particular, a K\"ahler manifold always admits balanced and pluriclosed metrics. The \textit{Fino--Vezzoni conjecture} \cite[Problem 3]{fino_2015} predicts the converse:

\begin{introconjecture}
    \label{first conj}
    If a compact complex manifold $Y$ admits both a balanced metric and a pluriclosed metric, then  $Y$ also admits a K\"ahler metric.
\end{introconjecture}
Note that a Hermitian metric which is simultaneously balanced and pluriclosed is necessarily K\"ahler \cite[Remark 1]{ivanov}.
Our main result verifies the first Fino--Vezzoni conjecture under the following homogeneity assumptions:
\begin{introassumptions}
    \label{assum}
    Let $Y$ be a compact complex manifold of the form $Y = \Gamma \backslash X$, where
    \begin{enumerate}
        \item $X$ is a connected complex manifold of complex dimension $n \geq 2$,
        \item \label{assum2}there exists a connected real Lie group $\G$ acting smoothly and transitively on $X$ via biholomorphisms, with compact isotropy subgroups, and 
        \item \label{assum3}$\Gamma \leq \G$ is a discrete subgroup acting freely and cocompactly on $X$.  
    \end{enumerate}
\end{introassumptions}
\begin{introtheorem}
    \label{thm main}
    Under Assumptions \ref{assum}, if $Y$ admits both a balanced metric and a pluriclosed metric, then $Y$ also admits a K\"ahler metric.
\end{introtheorem}
By specialising Theorem \ref{thm main} to the case when the isotropy  is trivial,  we have:
\begin{introcorollary}
    \label{main cor 1}
    Let $\G$ be a real Lie group, let $J$ be a left-invariant complex structure, and let $\Gamma$ be a cocompact lattice of $\G$. If the compact complex manifold $Y = (\Gamma \backslash \G, J)$ admits both a balanced metric and a pluriclosed metric, then $Y$  admits a flat K\"ahler metric. 
\end{introcorollary}
The flatness assertion follows from Hano's theorem \cite{hano} and Corollary \ref{cor sym}.
Previous results on balanced and pluriclosed metrics on Lie groups already imply Corollary \ref{main cor 1} when
 $\G$ is nilpotent \cite{fino_2016, arroyo_2022}, 
 $\G$ is two-step solvable \cite{fusi2026, Fribert_2025}, 
 $Y$ is an Oeljeklaus--Toma manifold \cite{otiman},
 $\G$ is almost abelian \cite{Fino_2023},
 $\G$ has an abelian ideal of codimension two \cite{Zheng_2024},
 $\G$ is a particular low-dimensional solvable  Lie group \cite{fino_2025,fino_2015},
 $\G$ is compact \cite{Fribert_2025,fino_2019},
 $\G$ is semisimple and $J$ is regular \cite{kwong2026, giusti_podesta_2023, lauret},
or  the first Bott-Chern class of $Y$ vanishes
\cite{fino2026}. 

On the other hand, by specialising Theorem \ref{thm main} to the case when $\Gamma$ is trivial, we have:

\begin{introcorollary}
    \label{main cor 2}
    Let $Y$ be a complex manifold, and suppose there exists a compact Lie group acting smoothly and transitively on $Y$ via biholomorphisms.  If $Y$ admits both a balanced metric and a pluriclosed metric, then $Y$ also admits a K\"ahler metric. In particular, $Y$ is a product of a complex torus and a complex generalised flag manifold.
\end{introcorollary}
The last sentence of Corollary \ref{main cor 2} follows from the Borel--Remmert theorem \cite[Chapter 3.9]{Akhiezer}.
 Under the additional assumption that the fundamental group of $Y$ is finite,  Corollary \ref{main cor 2} has already been established in \cite{fino_2019, podesta}.

A complex manifold $X$ admits a Lie group action as in \eqref{assum2} if and only if $X$ admits a Hermitian metric whose holomorphic isometry group acts transitively. Note that there exist  complex manifolds whose automorphism group acts transitively, but on which no finite-dimensional Lie group acts transitively by biholomorphisms \cite{kaup}.

Without the homogeneity assumption,   Conjecture  \ref{first conj} also holds   for compact complex manifolds in Fujiki class $\mathcal C$ \cite{Chiose_2014}, twistor spaces \cite{verbitsky},  compact complex threefolds whose balanced cone maps surjectively onto the Gauduchon cone \cite{chiose2}, some examples of non-K\"ahler Calabi--Yau threefolds \cite{yau,teng,picard},  compact BTP threefolds \cite{chen}, compact Vaisman manifolds \cite{otiman2},  some examples of compact LCK manifolds \cite{ornea}, Strominger K\"ahler-like manifolds \cite{zheng2}, and  some complex manifolds constructed via suspensions \cite{suspense}.

By a standard symmetrisation argument (see Corollary \ref{cor sym}), Theorem \ref{thm main} is an immediate consequence of Theorem \ref{thm tech 1}. In order to state this theorem, let us introduce some terminology. First, the space of \textit{$\G$-invariant Aeppli directions} is the set
$$\A := \left\{ \overline \partial \alpha + \partial \overline \alpha : \text{$\alpha$ is a $\G$-invariant  $(1,0)$-form on $X$}\right\}.$$
Next, for  a $\G$-invariant Hermitian metric $\omega_0$ on $X$ (not necessarily pluriclosed), the \textit{$\G$-invariant Aeppli slice $\S$ of $\omega_0$} is the following subset of $\G$-invariant Hermitian metrics on $X$:
$$\S := \left\{ \omega \in (\omega_0 + \A) : \omega > 0\right\}.$$
If $\omega_0$ is pluriclosed, then every element of $\S$ is also pluriclosed.

\begin{introtheorem}
    \label{thm tech 1}
    Under Assumptions \ref{assum}, suppose $Y$ admits a balanced metric. For any $\G$-invariant Hermitian metric $\omega_0$ on $X$, its $\G$-invariant Aeppli slice contains  
   a unique balanced metric $\omega_*$. In particular, if $\omega_0$ is pluriclosed, then $\omega_*$ is K\"ahler.
\end{introtheorem}

We now turn to the pluriclosed flow. Let $Y$ be any complex manifold, and let $\omega_0$ be a pluriclosed metric on $Y$. A smooth family of Hermitian metrics $\{\omega(t)\}_{t \in [0, T)}$ is called \textit{the pluriclosed flow starting at $\omega_0$} if 
$$\frac{\partial}{\partial t} \omega(t) = - (\rho^B_{\omega(t)})^{1,1}, \qquad \omega(0) = \omega_0,$$
 where $(\rho^B_{\omega(t)})^{1,1}$ denotes the $(1,1)$-part of the Bismut--Ricci form of $\omega(t)$.  
 
The Streets--Tian maximal existence time conjecture \cite[Conjecture 3.6]{streets_survey} predicts, in particular, that every pluriclosed flow on a compact complex manifold $Y$ with vanishing de Rham first Chern class $c_1(Y) = 0$ is immortal. Conjecture 1.2  in \cite{fino2026} asserts that if $Y$ also admits a balanced metric, then every pluriclosed flow must have a K\"ahler limit:
 \begin{introconjecture}
    \label{second conj}
     Let $Y$ be any compact complex manifold admitting a balanced metric with $c_1(Y) = 0$. Then the pluriclosed flow starting at any pluriclosed metric is immortal and converges smoothly to a K\"ahler metric.
 \end{introconjecture}

As evidence, Fino and Vezzoni prove Conjecture \ref{second conj} under the additional assumptions that $Y = (\Gamma \backslash \G, J)$, the first Bott--Chern class of $Y$ vanishes, and the pluriclosed flow is invariant \cite[Theorem 1.3]{fino2026}. Our next result generalises \cite[Theorem 1.3]{fino2026} to all $Y$ satisfying Assumptions \ref{assum} and drops the vanishing first Bott--Chern class assumption:

\begin{introtheorem}
    \label{thm tech 2}
    Under Assumptions \ref{assum}, suppose $Y$ admits a balanced metric and $c_1(Y) = 0$. For any $\G$-invariant pluriclosed metric $\omega_0$ on $X$, the pluriclosed flow starting at $\omega_0$ is immortal, remains in the $\G$-invariant Aeppli slice $\S$ of $\omega_0$, and converges smoothly to the unique K\"ahler metric $\omega_*$ in $\S$. Moreover, $\omega_*$ is flat. 
\end{introtheorem}


\subsection{Article outline and proof strategy}
In Section \ref{sec 2}, we prove Theorem \ref{thm tech 1}.  The key observation is that a $\G$-invariant Hermitian metric is balanced  precisely when it is a critical point of the volume functional $\V$ on its $\G$-invariant Aeppli slice  (Proposition \ref{prop_balanced_characterisation}). After symmetrising  a balanced metric on $Y$, we obtain a $\G$-invariant balanced metric on $X$, which forces each $\G$-invariant Aeppli slice to be bounded. Strict concavity of $\log \V$ supplies a unique maximiser $\omega_*$ on each slice (Lemma \ref{lem_cone}).

In Section \ref{sec 3}, we prove Theorem \ref{thm tech 2}. When $c_1(Y) = 0$, the pluriclosed flow restricts to an ODE on each $\G$-invariant Aeppli slice containing pluriclosed metrics. We show that $\L = \log \V$ is a proper strict Lyapunov function for this ODE, whose unique equilibrium point is the K\"ahler metric $\omega_*$ (Lemma \ref{lem ode}).

Finally, in Appendix \ref{sec A}, we construct the symmetrisation operator (Proposition \ref{prop sym operator}) and deduce Theorem \ref{thm main} from Theorem \ref{thm tech 1}.

\subsection{Acknowledgements} I thank James Stanfield, Elia Fusi, Anna Fino,  Luigi Vezzoni, Ramiro Lafuente, and Kyle Broder for reading  earlier versions of this manuscript and giving very helpful comments.
I also thank my supervisors Ramiro Lafuente and Kyle Broder for their continued encouragement and advice. 
I was supported by an Australian Government Research Training Program Scholarship.

\subsection{On the use of AI} The core idea of   Theorem \ref{thm tech 1} arose in discussions with  ChatGPT 5.6 Sol on the 25th of July, 2026.  This manuscript
does not contain AI-written text.

\section{Proof of Theorem \ref{thm tech 1}}
\label{sec 2}
Under Assumptions \ref{assum},  the real vector space of  $\G$-invariant real $(1,1)$-forms $$\H := \Omega_\R^{1,1}(X)^\G$$ on $X$ is finite-dimensional because the action of $\G$ on $X$ is transitive. Let $\P \subseteq \H$ denote the cone of $\G$-invariant Hermitian metrics on $X$, i.e.,
$$\P = \left\{  \omega \in \H : \omega > 0\right\}.$$
Henceforth, we identify $\G$-invariant forms on $X$ with their descents to $Y$ along the holomorphic covering map $X \rightarrow Y := \Gamma \backslash X$.
Let $\V: \H\rightarrow \R$ denote the volume functional
$$\V(\omega) := \frac{1}{n!} \int_Y \omega^n.$$
Let $\A \subseteq \H$ denote the subspace of \textit{$\G$-invariant Aeppli directions}
$$\A := \left\{ \overline \partial \alpha + \partial \overline \alpha: \alpha \in \Omega^{1,0}(X)^\G\right\} = \left\{(d \beta)^{1,1} : \beta \in \Omega^1_\R(X)^\G \right\}.$$
Here, $\Omega^{1,0}(X)^\G$ and $\Omega^1_\R(X)^\G$ denote the spaces of $\G$-invariant complex $(1,0)$-forms and real  $1$-forms on $X$, respectively.  For a $\G$-invariant Hermitian metric $\omega$ on $X$, its \textit{$\G$-invariant Aeppli slice} $\S$ is 
$$\S := (\omega + \A) \cap \P.$$

\begin{proposition}
    \label{prop_balanced_characterisation}
    A $\G$-invariant Hermitian metric $\omega$ on $X$ is balanced if and only if 
    $$d \V_\omega(\eta) = 0 \qquad \forall \eta \in \A,$$
    i.e., $\omega$ is a critical point of the restriction of $\V$ to the $\G$-invariant Aeppli slice of $\omega$.
\end{proposition}
\begin{proof}
    The differential  $d \V_\omega:  \H \rightarrow \R$ is given by
    $$d \V_\omega(\eta) = \frac{d}{dt}\bigg|_{t = 0} \V(\omega + t \eta)= \frac{1}{n!}\int_Y\frac{d}{dt}\bigg|_{t = 0} (\omega + t \eta)^n  = \frac{1}{(n-1)!} \int_Y \eta \wedge \omega^{n-1}.$$
    Now, suppose $\beta \in \Omega^1_\R(X)^\G$, and consider $\eta = (d \beta)^{1,1} \in \A$. We find 
    $$(n-1)! \,d \V_\omega(\eta) = \int_Y (d\beta)^{1,1} \wedge \omega^{n-1} = \int_Y d \beta \wedge \omega^{n-1} =  \int_Y  \beta \wedge d\omega^{n-1},$$
    where  the second equality holds because $\omega^{n-1}$ is an $(n-1,n-1)$-form, and the last equality holds by Stokes' theorem. 

    Clearly, $d \omega^{n-1} = 0$ implies $d \V_\omega(\A) = 0$. The converse holds because the pairing $$\Omega^1_\R(X)^\G \times \Omega^{2n-1}_\R(X)^\G \rightarrow \R, \qquad (\beta, \gamma) \mapsto \int_Y \beta \wedge \gamma$$
    is nondegenerate by Lemma \ref{pairing lemma}.
\end{proof}

\begin{remark}
    \label{rem gen}
    The proof above generalises to any compact complex manifold,  with no symmetry assumptions. In fact, a Hermitian metric $\omega$ on any compact complex manifold $Y$ is balanced if and only if it is a critical point of the volume functional $\V: \Omega_\R^{1,1}(Y) \rightarrow \R$ restricted to the full Aeppli slice of $\omega$ (cf. \cite[Proof of Corollary 3.7]{popovici}). 
\end{remark}

Now, choose:
\begin{itemize}
    \item A  background $\G$-invariant Hermitian metric $\omega_B$ on $X$.
    \item A base point $o \in X$.
\end{itemize}
Let $V$ denote the Hermitian inner product space  $(T^{1,0}_o X, \langle \cdot,\cdot \rangle)$, where  $\langle v,w \rangle := - \i \omega_{B}(v, \overline w)$ for all $v,w \in T^{1,0}_o X$.
Let $\K$ denote the compact isotropy subgroup of $\G$ at $o$, and 
let $\Herm(V)^\K$ denote the set of Hermitian operators on $V$ which commute with the isotropy representation $\K \rightarrow \GL(V)$. We have a real-linear isomorphism 
$$ \H \longleftrightarrow \Herm(V)^\K, \qquad \omega \mapsto A_\omega,$$
where $A_\omega$ is defined implicitly by $\langle A_\omega v,w\rangle= - \i \omega_o(v, \overline w)$ for $v,w \in V$. Since $\omega$ is $\G$-invariant and $X$ is $\G$-homogeneous, we can write $\omega^n = \det(A_\omega) \omega_B^n$. In particular,
$$\V(\omega) = \det(A_\omega) \,\V(\omega_B).$$
Once the data $(\omega_B,o)$ is fixed, we identify 
$$\H \cong \Herm(V)^\K \leq \Herm(V)$$
via the isomorphism above.  In particular, we set $\det, \tr: \H \rightarrow \R$ by
$\det(\omega) := \det(A_\omega)$ and $\tr(\omega) := \tr(A_{\omega})$, respectively. Since $\V|_\H = \det|_\H$   up to a constant, Proposition \ref{prop_balanced_characterisation} immediately implies the following:
\begin{corollary}
    \label{cor det}
    With the data $(\omega_B,o)$ fixed, a $\G$-invariant Hermitian metric $\omega$ on $X$ is balanced if and only if 
    $$(d \det)_\omega(\eta) = 0 \qquad \forall \eta \in \A.$$
    In particular, the background metric $\omega_B$ is balanced if and only if $\tr(\eta) = 0$ for all $\eta \in \A$.
\end{corollary}
The last sentence of Corollary \ref{cor det} follows because the  differential of $\det:\H \rightarrow \R$ at the background metric $\omega_B = \Id$ is  $\tr:\H \rightarrow \R$.

\begin{lemma}
    \label{lem_cone}
    Let $V$ be any finite-dimensional complex inner product space. Let $\H \leq \Herm(V)$ be a real subspace, and let
    $\P \subseteq \H$ be the set of positive-definite elements in $\H$.
    Let $\A \leq \H$ be a real subspace. Fix $A_0 \in \P$, and consider the slice
    $$\S := (A_0 + \A) \cap \P.$$
    Suppose $\S$ is bounded in $\H$. Then $\det|_\S:\S \rightarrow \R$ has a unique critical point $A_* \in \S$.
\end{lemma}
\begin{proof}[Proof of Lemma \ref{lem_cone}]
    Since $\S$ is bounded, its closure $\overline \S$ in $\H$ is compact, so $\det|_{\overline \S}: \overline \S \rightarrow \R$ attains a maximum at some $A_* \in \overline \S$. Every element of the boundary $\overline \S \backslash \S$ has determinant zero, while elements of $\S$ have positive determinant. Therefore, $A_*$ must lie in $\S$. In particular, $A_*$ is a critical point of $\det|_\S$.

    Let us show uniqueness of this critical point. First, the critical points of $\det|_\S$ are precisely the critical points of $\log \det|_\S$. For the sake of contradiction, suppose $A_{**} \in \S$ is another critical point, and set $D := A_{**} - A_* \neq 0$. Because $\S$ is convex, there exists an open interval $I$ containing $0$ and $1$ such that $A(t) := A_* + t D \in \S$ for any $t \in I$. Now, consider $f:I \rightarrow \R$ given by $f(t) := \log \det(A(t))$. Since $A_*$ and $A_{**}$ are critical points of $\log \det|_\S$, we have  $0 = f'(0) = f'(1)$. On the other hand, we find
    $$f''(t) = - \tr(A(t)^{-1} D A(t)^{-1} D) < 0.$$
    Thus, $f'$ is strictly decreasing, a contradiction.
\end{proof}

\begin{proof}[Proof of Theorem \ref{thm tech 1}]
Under Assumptions \ref{assum}, suppose $Y$ admits a balanced metric. By symmetrisation, $X$ admits a $\G$-invariant balanced metric $\omega_B$, which we use as the background metric.  We also choose a base point $o \in X$. 

Now, fix a $\G$-invariant Hermitian metric $\omega_0$ on $X$, and consider its $\G$-invariant Aeppli slice
$\S :=  (\omega_0 + \A) \cap \P.$
By Corollary \ref{cor det}, the balanced metrics in $\S$ are precisely the critical points of $\det|_{\mathcal S}: \mathcal S \rightarrow \R$. Since the background metric $\omega_B$ is balanced, Corollary \ref{cor det} also tells us that $\tr(\eta) = 0$ for all $\eta \in \A$. Therefore, if $\omega = \omega_0 + \eta \in \mathcal S$, then 
    $$\|A_\omega\|_{\operatorname{op}} = \lambda_{\max}(A_\omega)  \leq \tr(\omega) = \tr(\omega_0),$$
    where $\|\cdot\|_{\operatorname{op}}$ denotes the operator norm and $\lambda_{\max}$ denotes the largest eigenvalue.
    Thus, $\mathcal S$ is bounded, so Lemma \ref{lem_cone} supplies a unique critical point $\omega_*$ of $\det|_{\mathcal S}$. Corollary \ref{cor det} implies that $\omega_*$ is the unique balanced metric in $\S$. 

    Finally, for the last sentence of Theorem \ref{thm tech 1}, suppose $\omega_0$ is pluriclosed.  Because $\partial \overline \partial \eta = 0$ for all $\eta \in \A$, it follows that every metric in $\mathcal S$ is also pluriclosed. Therefore,  $\omega_*$ is also K\"ahler by \cite[Remark 1]{ivanov}.
\end{proof}

\section{Proof of Theorem \ref{thm tech 2}}
\label{sec 3}
Under Assumptions \ref{assum}, suppose $Y$ admits a balanced metric. Let $\omega_0 \in \P$ be any $\G$-invariant pluriclosed metric on $X$. Consider the $\G$-invariant Aeppli slice $\S := (\omega_0 + \A) \cap \P$ of $\omega_0$. 
By the proof of Theorem \ref{thm tech 1} above, $\S$  contains a unique K\"ahler metric, $\omega_*$.

Next, suppose the real first Chern class vanishes, $c_1(Y) = 0$. For each $\omega \in \S$, regarded as a Hermitian metric on $Y$, we have $[\rho^B_\omega]_{\text{dR}} = 2 \pi  c_1(Y) = 0$. In particular, there exists a real $1$-form $\beta$ on $Y$ such that $\rho^B_\omega = d \beta$. If $S: \Omega^\bullet(Y) \rightarrow \Omega^\bullet(X)^\G$ denotes the symmetrisation operator defined in Proposition \ref{prop sym operator}, then $$d S \beta = S d \beta = S \rho^B_\omega = \rho^B_\omega,$$
where the last equality holds because $\rho^B_\omega$ is already $\G$-invariant. Thus, after replacing $\beta$ with $S \beta$, we may assume $\beta \in \Omega^1(X)^\G$. In particular, $$(\rho^B_\omega)^{1,1} = (d \beta)^{1,1}  \in \A.$$
Therefore, the pluriclosed flow restricts to an ODE on $\S$.

Now, let $\X: \S \rightarrow \A$ denote the smooth vector field corresponding to the pluriclosed flow on $\S$, i.e.,
$$\X_\omega = - (\rho^B_\omega)^{1,1}.$$
Choose $\omega_*$ as a background metric defining $\det:\S \rightarrow \R$, so that $\omega^n = \det(\omega) \, \omega_*^n$. Let $\L:\S \rightarrow \R$ be the smooth function 
$$\L(\omega) := \log \det  (\omega).$$
\begin{lemma}
    \label{lem ode}
    The K\"ahler metric $\omega_* \in \S$ is an equilibrium of the pluriclosed flow, i.e., $$\X_{\omega_*} = 0.$$ 
    The function $\L = \log \det|_\S$ is a proper strict Lyapunov function for the pluriclosed flow $\X$ at the equilibrium $\omega_*$. In other words:
    \begin{enumerate}
        \item  \label{ode3}$\L: \S \rightarrow \R$ attains a unique global maximum at $\omega_*$.
        \item  \label{ode2}$d \L_\omega(\X_\omega) \geq 0$ for any $\omega \in \S$, and $d \L_\omega(\X_\omega) = 0$ if and only if $\omega = \omega_*$.
        \item  \label{ode4}For every $c \in \R$, the superlevel set $\{\omega \in \S : \L(\omega) \geq c\}$ is compact.
    \end{enumerate}
\end{lemma}
\begin{proof}
    Since the pluriclosed flow preserves the K\"ahler condition and $\omega_*$ is the unique K\"ahler metric in $\S$, the flow starting at $\omega_*$ is constant. In particular, $\omega_*$ is  Ricci-flat. Since $\omega_*$ is also homogeneous, the Alekseevskii--Kimelfeld theorem \cite[Theorem 7.61]{besse} implies that $\omega_*$ is flat.

    The proof of Theorem \ref{thm tech 1} implies \eqref{ode3}.   Next, let $\omega:I \rightarrow \S$ be an integral curve of $\X$. We find 
    $$d\L_{\omega(t)}(\X_{\omega(t)}) = \frac{d}{dt} \L(\omega(t)) = \frac{d}{dt} \log \det(\omega(t)) = |T_{\omega(t)}|^2_{\omega(t)} \geq 0.$$
    The last equality follows from \cite[Lemma 6.1]{streets} and the fact that $\omega_*$ is flat. (Note that the Laplacian term in \cite[Lemma 6.1]{streets} vanishes  because $\log \det (\omega)$ is spatially constant on $X$, thanks to  homogeneity.) Here $T_{\omega(t)}$ is the torsion of the Chern connection of $\omega(t)$.  Property \eqref{ode2} follows because  $T_{\omega(t)} = 0$ if and only if $\omega(t)$ is K\"ahler, and $\omega_*$ is the only K\"ahler metric in $\S$.

    The proof of Theorem \ref{thm tech 1} implies that $\S$ is bounded in the finite-dimensional vector space $\H$.
    Property \eqref{ode4} follows because $$\left\{\omega \in \S : \L(\omega) \geq c \right\} = \left\{\omega \in \overline\S : \det(\omega) \geq e^c \right\}$$ is closed and bounded in $\H$.
\end{proof}

Since $\L:\S \rightarrow \R$ is a proper strict Lyapunov function at $\omega_*$, standard ODE theory implies that  every trajectory of the pluriclosed flow in $\S$ is immortal and converges to $\omega_*$. Moreover, $\omega_*$ is globally asymptotically stable. This completes the proof of Theorem \ref{thm tech 2}.

\appendix
\section{The symmetrisation operator}
\label{sec A}
In this appendix, let $\Omega^k(X)$ and $\Omega^k(Y)$ denote the real $k$-forms on $X$ and $Y$, respectively.
Using the holomorphic covering map $X \rightarrow \Gamma \backslash X = Y$, we identify the $\G$-invariant $k$-forms $\Omega^k(X)^\G$ with their descents to $\Omega^k(Y)$, and we identify the $k$-forms $\Omega^k(Y)$ with their lifts to $\Omega^k(X)$, which are precisely the $\Gamma$-invariant forms in  $\Omega^k(X)$.

\begin{proposition}
    \label{prop sym operator}
    Under Assumptions \ref{assum}, for every $k = 0,\ldots, N := 2n$, there exists a unique linear map $S: \Omega^k(Y) \rightarrow \Omega^k(X)^\G$ such that for every $\alpha \in \Omega^k(Y)$, we have
    \begin{align}
        \label{def S}
        \int_Y \alpha \wedge \beta = \int_Y S \alpha \wedge \beta \qquad \text{for all $\beta \in \Omega^{N-k}(X)^\G$}.
    \end{align}
    Moreover, the map $S: \Omega^{\bullet}(Y) \rightarrow \Omega^{\bullet}(X)^\G$ satisfies the following properties:
    \begin{enumerate}[(i)]
        \item \label{sym0} If $\alpha$ is already $\G$-invariant, then $S \alpha = \alpha$.
        \item \label{sym1}$S$ commutes with the exterior derivative $d$.
        \item \label{sym2}There exists a bi-invariant volume form $\nu$ on $\G$ such that for every $\alpha \in \Omega^k(Y)$ and $x \in X$, we have 
        \begin{align}
            \label{sym integral}
            (S\alpha)_x = \int_{\Gamma \backslash G} (L_g^* \alpha)_x\;  \nu(\Gamma g),
        \end{align}
        where $L_g:X \rightarrow X$ is the map $x \mapsto g \cdot x$.
        \item \label{sym3}The complexification of $S$ sends $(p,q)$-forms to $(p,q)$-forms.
        \item \label{sym4}$S$ commutes with the Dolbeault operators $\partial$ and $\overline \partial$.
        \item \label{sym5}$S$ sends positive $(p,p)$-forms to positive $(p,p)$-forms.
    \end{enumerate}
\end{proposition}
We prove Proposition \ref{prop sym operator} at the end of this appendix. 

\begin{remark}
    When $X$ is a Lie group with an invariant complex structure, the symmetrisation procedure in Proposition \ref{prop sym operator} is well known: see \cite[Proposition 2.2]{kwong2026},
\cite[Section 2]{fino_grant_2004}, \cite[Prop 3.6]{ugarte}, \cite[Theorem 7]{belgun} and  \cite[Propositions 2.1 and 2.3]{giusti_podesta_2023}.
\end{remark}
\begin{corollary}
    \label{cor sym}
    Under Assumptions \ref{assum}, the following are equivalent:
    \begin{enumerate}
        \item \label{cor1}$Y$ admits a K\"ahler/balanced/pluriclosed metric.
        \item \label{cor2}$X$ admits a $\G$-invariant K\"ahler/balanced/pluriclosed metric.
    \end{enumerate}
\end{corollary}
\begin{proof}[Proof of Corollary \ref{cor sym}, assuming Proposition \ref{prop sym operator}]
    We show that \eqref{cor1} implies \eqref{cor2}. The other direction is trivial.
    Let $ \omega \in \Omega^{1,1}_\R(Y)$ be any Hermitian metric on $Y$. Property \eqref{sym5} implies that $S \omega$ is a $\G$-invariant Hermitian metric on $X$. If $\omega$ is K\"ahler/pluriclosed, then \eqref{sym1} and \eqref{sym4} imply that $S \omega$ is also K\"ahler/pluriclosed, respectively.

    Finally, suppose $\omega$ is balanced, and set $\Omega := \omega^{n-1}$, which is a positive $(n-1,n-1)$-form satisfying $d \Omega = 0$. Then \eqref{sym5} implies that $S \Omega$ is a $\G$-invariant positive $(n-1,n-1)$-form on $X$, and \eqref{sym1} implies that $d S \Omega = S d \Omega = 0$. By \cite[Section 4]{michelsohn_1982},  there exists a unique Hermitian metric $\omega_1$ on $X$ such that $\omega_1^{n-1} = S \Omega$, i.e., $\omega_1$ is balanced. Moreover, 
    $$(L_g^* \omega_1)^{n-1} = L_g^* \omega_1^{n-1} = L_g^* S \Omega = S \Omega,$$
    so uniqueness of the root $\omega_1$ implies that $\omega_1$ is $\G$-invariant.
\end{proof}

\begin{lemma}
    \label{pairing lemma}
    Under Assumptions \ref{assum}, the following pairing is perfect:
    \begin{align}
        \label{pairing eq}
        \Omega^k(X)^\G \times \Omega^{N-k}(X)^\G \longrightarrow \R, \qquad (\gamma, \beta) \longmapsto \int_Y \gamma \wedge \beta.
    \end{align}
\end{lemma}
\begin{proof}[Proof of Lemma \ref{pairing lemma}]
    Since the isotropy subgroups of the action of $\G$ on $X$ are compact, there exists a $\G$-invariant Riemannian metric $g$ on $X$ by a standard averaging procedure. Let $*_g$ denote the Hodge star operator of $g$, which sends $\G$-invariant forms to $\G$-invariant forms. In particular, $\Omega^k(X)^\G$ and $\Omega^{N-k}(X)^\G$ have the same finite dimension. Let $\mu_g$ denote the Riemannian volume form of $g$. Fix $\alpha \in \Omega^k(X)^\G$, and suppose $\alpha \neq 0$. Then
    $$\int_Y \alpha \wedge (*_g \alpha)  = \int_Y |\alpha |_g^2 \, \mu_g >  0.$$
    Therefore, the induced linear map $\Omega^k(X)^\G  \rightarrow \left(\Omega^{N-k}(X)^\G \right)^*$ is injective. Because $\Omega^k(X)^\G$ and $\Omega^{N-k}(X)^\G$ have the same finite dimension, it follows that the pairing is perfect.
\end{proof}
\begin{proof}[Proof of Proposition \ref{prop sym operator}]
    Throughout the proof, $\alpha$ denotes a $k$-form on $Y$.
    
    First, let us show existence and uniqueness.  Observe that $\alpha$ defines a linear functional $F_\alpha:\Omega^{N-k}(X)^\G \rightarrow \R$ by $\beta \mapsto  \int_Y \alpha \wedge \beta$. Since the pairing \eqref{pairing eq} is perfect,
     there exists a unique element $S\alpha \in \Omega^k(X)^\G$ which induces $F_\alpha$, i.e., for every $\beta \in\Omega^{N-k}(X)^\G$, we have
    $$\int_Y S \alpha \wedge \beta = F_{\alpha}(\beta) =\int_Y \alpha \wedge \beta. $$
    
    Property  \eqref{sym0} follows from the definition of $S$. To show \eqref{sym1}, suppose $k \leq N -1$ and observe that for $\beta \in \Omega^{N-k -1}(X)^\G$, we have
    $$\int_Y (d S \alpha) \wedge  \beta = (-1)^{k+1} \int_Y S\alpha \wedge d \beta = (-1)^{k+1} \int_Y \alpha \wedge d \beta = \int_Y d \alpha \wedge \beta = \int_Y (S d \alpha) \wedge \beta.$$
    The second and fourth equalities follow from the definition of $S$. The first and third equalities follow from Stokes' theorem. 

    Let us prove \eqref{sym2}. Let us begin by showing that $\G$ admits a bi-invariant volume form. To this end, fix a base point $o \in X$, and let $\K$ be the isotropy subgroup of $\G$ at $o$, which is compact by assumption. Since $\Gamma$ acts freely on $X = \G/ \K$ from the left, it follows that $\K$ acts freely on $\Gamma \backslash \G$ from the right. Thus, $\Gamma \backslash \G \rightarrow \Gamma \backslash \G / \K = Y$ is a principal $\K$-bundle with compact base and fibres. Therefore, the total space $\Gamma \backslash \G$ is also compact. In particular, $\Gamma$ is a cocompact lattice of $\G$, so $\G$ is unimodular \cite[Lemma 6.2]{milnor}. Since $\G$ is connected and unimodular, it admits a bi-invariant volume form $\nu$.

    Next, set $\mu := \pi_* \nu$, where $\pi_*: \Omega^{\dim_\R \G}(\G) \rightarrow  \Omega^N(X)$ denotes integration along the fibres of the principal $\K$-bundle $\pi: \G \rightarrow \G / \K = X$. Since $\nu$ is left-invariant, the $\G$-equivariance of $\pi: \G \rightarrow X$ implies that $\mu$ is a $\G$-invariant volume form on $X$. We identify $\nu$ and $\mu$ with their descents to $\Gamma \backslash G$ and $Y$, respectively. Rescale $\nu$, and hence $\mu := \pi_* \nu$, so that 
    \begin{align}
        \label{int equal 1}
        \int_Y \mu  =1.
    \end{align}
    
    Let $(S' \alpha)_x$ denote the integral on the right-hand side of \eqref{sym integral}. We wish to show that $S' \alpha = S\alpha$. First, note that the definition of $(S' \alpha)_x$ makes sense because the integrand $\G \rightarrow \Lambda^k T_x^* X$,  $g \mapsto (L_g^* \alpha)_x$ is smooth and $\Gamma$-invariant. Indeed, if $\gamma \in \Gamma$, then $L_{\gamma g}^* \alpha = L_g^* L_\gamma^* \alpha = L_g^* \alpha$, since $\alpha$ is $\Gamma$-invariant. Next, let us show that $S' \alpha \in \Omega^k(X)^\G$; it suffices to show $S' \alpha$ is left-invariant. Indeed, for $h \in \G$ and $x \in X$, we find 
    $$(L^*_h S' \alpha)_x = \int_{\Gamma \backslash \G} (L_h^* L_g^* \alpha)_x \,\nu =\int_{\Gamma \backslash \G} (L_{gh}^* \alpha)_x \,\nu = \int_{\Gamma \backslash \G} (L_{g}^* \alpha)_x \,\nu = (S' \alpha)_x,$$
    where we used the right-invariance of $\nu$ in the second-last equality.

    Now, fix $\beta \in \Omega^{N-k}(X)^\G$. Since $\mu$ is a $\G$-invariant volume form on $X$, there exists a $\Gamma$-invariant smooth function $f:X \rightarrow \R$ such that 
    $\alpha \wedge \beta = f \mu.$ At the base point $o \in X$, we have
    \begin{align*}
        (S' \alpha \wedge \beta)_o &= \int_{\Gamma \backslash \G} (L_g^* \alpha \wedge \beta)_o \, \nu =  \int_{\Gamma \backslash \G}L_g^* (\alpha \wedge \beta)_o \, \nu=  \int_{\Gamma \backslash \G}L_g^* (f \mu)_o \, \nu  \\&=  \left( \int_{\Gamma \backslash \G} (f \circ \pi) \ \nu\right) \mu_o = \left( \int_Y f \, \mu\right) \mu_o = \left( \int_Y \alpha \wedge \beta\right) \mu_o,
    \end{align*}
    where the second-last equality follows because $\mu = \pi_* \nu$. Since $S' \alpha \wedge \beta$ and $\left( \int \alpha \wedge \beta\right) \mu$ are both $\G$-invariant and agree at $o$, they are equal. Therefore, 
    $$\int_Y S' \alpha \wedge \beta = \left( \int_Y \alpha \wedge \beta\right)\int_Y \mu  = \int_Y \alpha \wedge \beta,$$
    where the last equality follows from \eqref{int equal 1}. The uniqueness property of $S$ implies that $S\alpha = S'\alpha$, as desired.

    To see \eqref{sym3}, suppose $\alpha$ is a $(p,q)$-form on $Y$, lifted to $X$. Since each $L_g:X \rightarrow X$ is a biholomorphism, it follows that each $L_g^* \alpha$ is a $(p,q)$-form on $X$. The integral formula \eqref{sym2} implies that $S \alpha$ is also a $(p,q)$-form on $X$. Property \eqref{sym4} follows immediately from \eqref{sym1} and \eqref{sym3}. Finally, for \eqref{sym5}, suppose $\alpha$ is a positive $(p,p)$-form on $Y$. Again, since each $L_g$ is a biholomorphism, it follows that $L_g^* \alpha$ is a positive $(p,p)$-form on $X$, and the integral formula \eqref{sym2} implies that $S \alpha$ is a positive $(p,p)$-form.
\end{proof}

\bibliographystyle{amsalpha} 
\bibliography{references}

\end{document}